\documentclass[pdflatex,sn-mathphys-num]{sn-jnl}

\usepackage{graphicx}%
\usepackage{multirow}%
\usepackage{amsmath,amssymb,amsfonts}%
\usepackage{amsthm}%
\usepackage{mathrsfs}%
\usepackage[title]{appendix}%
\usepackage{xcolor}%
\usepackage{textcomp}%
\usepackage{manyfoot}%
\usepackage{booktabs}%
\usepackage{algorithm}%
\usepackage{algorithmicx}%
\usepackage{algpseudocode}%
\usepackage{listings}%
\usepackage{tikz}
\usetikzlibrary{matrix,positioning}
\usetikzlibrary{calc}

\theoremstyle{thmstyleone}%
\newtheorem{theorem}{Theorem}
\newtheorem{proposition}[theorem]{Proposition}%
\newtheorem{lemma}[theorem]{Lemma}

\theoremstyle{thmstyletwo}%
\newtheorem{example}{Example}%
\newtheorem{remark}{Remark}%
\newtheorem{corollary}[theorem]{Corollary}
\theoremstyle{thmstylethree}%
\newtheorem{definition}{Definition}%

\begin{document}

\title{$F$-Contraction with an Auxiliary Function and Its Application to Terrain-Following Airplane Navigation}


\author*[1,2]{\fnm{Irom} \sur{Shashikanta Singh}}\email{shashikanta.phd.math@manipuruniv.ac.in}

\author[2,1]{\fnm{Yumnam} \sur{Mahendra Singh}}\email{ymahemit@rediffmail.com}
\equalcont{These authors contributed equally to this work.}

\affil*[1]{\orgdiv{Department of mathematics}, \orgname{Manipur University}, \orgaddress{\street{Canchipur}, \city{Imphal west}, \postcode{795001}, \state{Manipur}, \country{India}}}

\affil[2]{\orgdiv{Department of basic science and humanities}, \orgname{Manipur Institute of Technology, A constituent college of Manipur university},  \orgaddress{\street{Canchipur}, \city{Imphal west}, \postcode{795001}, \state{Manipur}, \country{India}}}


\abstract{This paper aims to integrate the concepts of \( F \)-contraction and \textit{\( S^B \)-contraction} within the context of super metric spaces. Specifically, we introduce the concepts of \textit{\( S^F \)-contraction} and \textit{Bianchini \( S^F \)-contraction}. We demonstrate that these new concepts are genuine generalizations of \textit{\( S^B \)- and \( S^K \)-contractions} by providing nontrivial examples. Furthermore, we establish the existence and uniqueness of fixed points for mappings that satisfy these contractions. Lastly, we apply our findings to a model describing an airplane capable of automatically following a terrain.}




\maketitle

\section{Introduction and Preliminaries}

Metric fixed point theory is one of the most active and influential branches of nonlinear functional analysis. 
Following the celebrated \textit{Banach contraction principle} \cite{banach}, numerous results and developments have appeared in the literature over the past decades. 
In general, progress in metric fixed point theory can be viewed from two main perspectives. 
The first approach focuses on weakening or modifying the contractive conditions imposed on mappings, while the second approach studies extensions of the underlying space. 
In many cases, further developments can be regarded as attempts to overcome the limitations of previously known results.

The \textit{Banach contraction principle} guarantees the existence and uniqueness of a fixed point for a self-mapping on a complete metric space. 
More precisely, let $(W,\eth)$ be a complete metric space and let $\Upsilon  : W \to W$ be a mapping satisfying
\begin{equation}\label{banach}
\eth(\Upsilon \xi,\Upsilon \zeta) \le \beta\, \eth(\xi,\zeta),
\end{equation}
for all $\xi,\zeta \in W$, where $\beta \in [0,1)$. 
Then, $\Upsilon $ admits a unique fixed point in $W$. 
However, one limitation of this principle is that any mapping $\Upsilon $ satisfying \eqref{banach} is necessarily continuous on $W$.

In 1968, Kannan \cite{kannan} addressed this limitation by introducing another important class of contractive mappings that differs from the \textit{Banach contraction}. 
This type of contraction does not require the continuity of the mapping and has attracted considerable attention in fixed point theory.

A mapping $\Upsilon  : W \to W$ is called a \emph{Kannan contraction} if, for all $\xi,\zeta \in W$,
\begin{equation}\label{kannan}
\eth(\Upsilon \xi,\Upsilon \zeta) \le \frac{\beta}{2} \big(\eth(\xi,\Upsilon \xi ) + \eth(\zeta ,\Upsilon \zeta)\big),
\end{equation}
where $\beta \in [0,1)$.

The integration of these two entirely independent contractions—\textit{Banach} and \textit{Kannan}—was a significant question that required attention. Exploring this potential connection is crucial for advancing our understanding in the field. S. Reich \cite{reich} addressed this issue by introducing a new type of contraction that generalizes both the \textit{Banach} and \textit{Kannan contraction}.
\begin{theorem}\cite{reich}
Let $(W, \eth)$ be a complete metric space, and $\Upsilon $ be a  mapping from $W$ into itself, satisfying the following condition,
\begin{equation}
\eth(\Upsilon \xi, \Upsilon \zeta) \leq a \eth(\xi, \Upsilon \xi ) + b \eth(\zeta , \Upsilon \zeta) + c \eth(\xi, \zeta), \text{ for all }  \xi,\zeta \in W.
\end{equation}
where $a$, $b$, and $c$ are non-negative numbers satisfying $a + b + c < 1$.
Then, $\Upsilon $ has a unique fixed point.		
\end{theorem}

In 1972, Bianchini \cite{bianchini} introduced a new generalization Kannan contraction inspired by the classical inequality $\frac{a+b}{2} \le \max \{a,b\}$.
\begin{theorem}\cite{bianchini}
    Let $(W, \eth)$ be a complete metric space, and $\Upsilon $ be a  mapping from $W$ into itself, satisfying the following condition,
\begin{equation}
\eth(\Upsilon \xi, \Upsilon \zeta) \leq \beta \max\{\eth(\xi,\Upsilon \xi ),\eth(\zeta ,\Upsilon \zeta)\}, \text{ for all }  \xi,\zeta \in W.
\end{equation}
where $\beta \in (0,1)$.
Then, $\Upsilon $ has a unique fixed point.	
\end{theorem}

Furthermore, Wardowski \cite{wardowski} introduced a special class of contractions known as \emph{$F$-contractions}. 
This concept is based on a real-valued function $F : (0,\infty) \to \mathbb{R}$ satisfying the following conditions:
\begin{enumerate}
    \item[(W1)] The function $F$ is monotone increasing in a strict sense; namely, for any $\lambda,\beta>0$, if $\lambda<\beta$, then $F(\lambda)<F(\beta)$.

    \item[(W2)] Let $\{\lambda_n\}$ be a sequence of positive real numbers. Then
    \[
    \lambda_n \to 0 \text{ as } n\to\infty 
    \quad \text{if and only if} \quad
    F(\lambda_n) \to -\infty \text{ as } n\to\infty.
    \]

    \item[(W3)] There exists a number $k\in(0,1)$ for which
    \[
    \lim_{\lambda\to 0^{+}} \lambda^{k}F(\lambda)=0.
    \]
\end{enumerate}

The set of all such functions is denoted by $\mathcal{F}$.

Let $(W,\eth)$ be a metric space. A mapping $\Upsilon  : W \to W$ is called an 
\emph{$F$-contraction} if there exists $\omega   > 0$ such that
\[
\eth(\Upsilon \xi,\Upsilon \zeta) > 0
\;\Rightarrow\;
\omega   + F\big(\eth(\Upsilon \xi,\Upsilon \zeta)\big) \le F\big(\eth(\xi,\zeta)\big),
\]
for all $\xi,\zeta \in W$.

The following are the definitions of the \textit{Picard operator} and asymptotically regularity.
\begin{definition}\cite{picard}
Let $(W,\eth)$ be a metric space and $\Upsilon : W\to W$ be a self-mapping. Then, $\Upsilon $ is called the \textit{Picard operator} if $\Upsilon $ has a unique fixed point $u \in W$ and for any $\xi \in W$, the sequence of iterates $\{\Upsilon^n\xi\}$ converges to $u$.	
\end{definition}
\begin{definition} \cite{asym}
Let $(W,\eth)$ be a metric space. The mapping
$\Upsilon $ is said to be asymptotically regular at $\xi_0 \in W$ if
$\lim_{n \to \infty} \eth(\Upsilon^n{\xi_0}, \Upsilon^{n+1}{\xi_0}) = 0$
and  $\Upsilon :W\rightarrow W$ is said to be asymptotically regular  in $W$ if $\Upsilon $ is asymptotically regular at each $\xi \in W$.
\end{definition}

In 2020, Batra et al. \cite{batra} introduced the concept of \textit{Kannan $F$-contraction} and proved the existence and uniqueness of a fixed point in complete metric spaces.

\begin{definition}
Let $(W,\eth)$ be a metric space and $F \in \mathcal{F}$. 
A mapping $\Upsilon :W\to W$ is said to be a \emph{\textit{Kannan $F$-contraction}} if the following conditions hold:
\begin{enumerate}
    \item[(i)] $\Upsilon \xi \neq \Upsilon \zeta \implies \Upsilon \xi  \neq \xi \text{ or } \Upsilon \zeta \neq \zeta$;
    \item[(ii)] 
    \[
    \omega   + F\left(\eth(\Upsilon \xi,\Upsilon \zeta)\right)
    \le 
    F\left(\frac{\eth(\xi,\Upsilon \xi )+\eth(\zeta ,\Upsilon \zeta)}{2}\right),
    \]
\end{enumerate}
for all $\xi,\zeta \in W$ with $\Upsilon \xi \neq \Upsilon \zeta$.
\end{definition}

\begin{theorem}\cite{batra}
Let $(W,\eth)$ be a complete metric space. 
If $\Upsilon :W\to W$ is a \textit{Kannan $F$-contraction}, then $\Upsilon $ is a \textit{Picard operator}.
\end{theorem}

Recently, Bessem \cite{bessem} generalized the \textit{Banach contraction principle} by using an auxiliary function $S:W \times W \to W$.

\begin{definition}\label{sb definition}\cite{bessem}
Let $(W,\eth)$ be a metric space and $S:W\times W \to W$ be a given mapping. 
A mapping $\Upsilon : W \to W$ is said to be an $S^B$-contraction if there exists $\beta \in (0,1)$ such that
\begin{equation}\label{sb contraction}
\eth(\Upsilon \xi,S(\Upsilon \xi,\Upsilon \zeta)) + \eth(\Upsilon \zeta,S(\Upsilon \xi,\Upsilon \zeta)) 
\le 
\beta \big(\eth(\xi,S(\xi,\zeta)) + \eth(\zeta ,S(\xi,\zeta))\big)
\end{equation}
for all $\xi,\zeta \in W$.
\end{definition}

If $S(\xi,\zeta)=\xi$, then the \textit{$S^B$-contraction} reduces to the \textit{Banach contraction}.

\begin{theorem}\cite{bessem}
Let $(W,\eth)$ be a complete metric space and $S:W\times W \to W$ be a given mapping. 
If $\Upsilon  : W \to W$ is an \textit{$S^B$-contraction}, then $\Upsilon $ is a \textit{Picard operator} provided that, for all $u,v \in W$,
\[
\lim_{n \to \infty} \eth(\Upsilon^n u, v) = 0 
\quad \Longrightarrow \quad 
\lim_{n \to \infty} \eth(\Upsilon(\Upsilon^n u), \Upsilon  v) = 0.
\]
\end{theorem}

\begin{definition}\label{sk definition}\cite{bessem}
Let $(W,\eth)$ be a metric space and $S:W\times W \to W$ be a given mapping. 
A mapping $\Upsilon : W \to W$ is said to be an \textit{$S^K$-contraction} if there exists $\beta \in (0,1)$ such that
\begin{equation}\label{sk contraction}
\begin{split}
\eth(\Upsilon \xi,S(\Upsilon \xi,\Upsilon \zeta)) + \eth(\Upsilon \zeta,S(\Upsilon \xi,\Upsilon \zeta)) \le 
&\frac{\beta}{2} \Big(\eth(\xi,S(\xi,\Upsilon \xi )) + \eth(\Upsilon \xi,S(\xi,\Upsilon \xi )) \\
&+ \eth(\zeta ,S(\zeta,\Upsilon \zeta)) + \eth(\Upsilon \zeta,S(\zeta,\Upsilon \zeta))\Big),
\end{split}
\end{equation}
for all $\xi,\zeta \in W$.
\end{definition}

If $S(\xi,\zeta)=\xi$, then the \textit{$S^K$-contraction} reduces to the \textit{Kannan contraction}.

\begin{theorem}\cite{bessem}
Let $(W,\eth)$ be a complete metric space and $S:W\times W \to W$ be a given mapping with $S(\xi,\xi)=\xi$. 
If $\Upsilon  : W \to W$ is an \textit{$S^K$-contraction}, then $\Upsilon $ is a \textit{Picard operator} provided that, for all $u,v \in W$,
\[
\lim_{n \to \infty} \eth(\Upsilon^n u, v) = 0 
\quad \Longrightarrow \quad 
\lim_{n \to \infty} \eth(\Upsilon(\Upsilon^n u), \Upsilon  v) = 0.
\]
\end{theorem}

Over time, metric spaces have been generalized in several directions. 
Some well-known generalizations include quasi-metric spaces, $b$-metric spaces, symmetric spaces, dislocated metric spaces, partial metric spaces, modular metric spaces, cone metric spaces, ultrametric spaces, and various hybrid structures derived from these concepts.

Recently, another interesting generalization known as a \emph{super-metric space} has been introduced. 
Every metric space satisfies the axioms of a super-metric space, although the triangle inequality is not necessarily required in this setting.

\begin{definition}\cite{super metric}\label{sms}
Let $W$ be a nonempty set and let $\eth : W \times W \to [0,\infty)$ be a function. 
Then $\eth$ is called a \emph{super-metric} on $W$ if the following conditions hold:
\begin{enumerate}
    \item  $\eth(\xi,\zeta)=0$ iff $\xi =\zeta $ for all $\xi,\zeta \in W$.
    \item $\eth(\xi,\zeta)=\eth(\zeta ,\xi)$ for all $\xi,\zeta \in W$.
    \item There exists a constant $s \ge 1$ such that for every $\zeta \in  W$, there exist two distinct sequences $(\xi_n)$ and $(\zeta_n)$ in $W$ satisfying
    \[
    \eth(\xi_n,\zeta_n) \to 0 \quad \text{as } n \to \infty,
    \]
    such that
    \[
    \limsup_{n\to\infty} \eth(\zeta_n,\zeta)
    \le
    s \, \limsup_{n\to\infty} \eth(\xi_n,\zeta).
    \]
\end{enumerate}
Then, $(W,\eth,s)$ is called a \emph{super-metric space}.
\end{definition}
Axiom 3 of Definition \ref{sms} is called super-metric triangle inequality.
\begin{example}\cite{super metric}
Let $W=[0,\infty)$ and define $\eth : W \times W \to [0,\infty)$ by
\[
\eth(\xi,\zeta)=
\begin{cases}
\dfrac{\xi+\zeta }{1+\xi+\zeta }, & \xi \neq \zeta,\; \xi\neq 0,\; \zeta \neq 0, \\[6pt]
0, & \xi =\zeta , \\[6pt]
\max\!\left\{\dfrac{\xi }{2},\,\dfrac{\zeta }{2}\right\}, & \text{otherwise}.
\end{cases}
\]
Then $(W,\eth,s)$ is a super-metric space.
\end{example}

\begin{definition}\cite{super metric1}
Let $(W,\eth,s)$ be a supermetric space and let $\{\xi_n\}$ be a sequence in $W$.
\begin{enumerate}
    \item[(i)] The sequence $\{\xi_n\}$ is said to converge to $\xi \in W$ 
    if and only if
    \[
    \lim_{n \to \infty} \eth(\xi_n,\xi) = 0.
    \]

    \item[(ii)] The sequence $\{\xi_n\}$ is said to be a Cauchy sequence in $W$ 
    if and only if
    \[
    \lim_{n \to \infty} \sup \{ \eth(\xi_n,\xi_p) : p > n \} = 0.
    \]
\end{enumerate}
\end{definition}

\begin{proposition}\cite{super metric1}
On a supermetric space $(W,\eth,s)$, the limit of a convergent sequence is unique.
\end{proposition}

\begin{definition}\cite{super metric1}
A supermetric space $(W,\eth,s)$ is said to be complete if every Cauchy sequence in $W$ is convergent in $W$.
\end{definition}

Motivated by these recent developments, this paper aims to combine the concepts of \( F \)-contraction and \( S^B \)-contraction within the context of super-metric spaces. Specifically, we introduce the notions of \( S^F \)-contraction and Bianchini \( S^F \)-contraction. We demonstrate that these new concepts are genuine generalizations of \( S^B \)- and \( S^K \)-contractions, respectively, by providing nontrivial examples. Furthermore, we establish the existence and uniqueness of fixed points for mappings that satisfy these contractions. Lastly, we apply our results to a model that describes an airplane capable of automatically following the terrain.

\section{\texorpdfstring{\textit{$S^F$-contraction}: A generalization of \textit{$S^B$-contraction}}
{SF-contraction: A generalization of SB-Contraction}}

In this section, we introduce the concept of \textit{$S^F$-contraction} and prove that any mapping satisfying \textit{$S^F$-contraction} in a complete super-metric space is a \textit{Picard operator}.
\begin{definition}\label{sf definition}
    Let $(W,\eth,s)$ be a super-metric space and $S:W\times W \to W$ be a given mapping. A mapping $\Upsilon : W \to W$ is said to be $S^F$- contraction if $F \in \mathcal{F}$ and there is some $\omega   > 0$ such that
    \begin{equation}\label{sf contraction}
        \omega  + F\left(\eth(\Upsilon \xi,S(\Upsilon \xi,\Upsilon \zeta))+\eth(\Upsilon \zeta,S(\Upsilon \xi,\Upsilon \zeta))\right) \le F\left(\eth(\xi,S(\xi,\zeta))+\eth(\zeta ,S(\xi,\zeta))\right)
    \end{equation}
    for all $\xi,\zeta \in W$ with $\eth(\Upsilon \xi,S(\Upsilon \xi,\Upsilon \zeta))+\eth(\Upsilon \zeta,S(\Upsilon \xi,\Upsilon \zeta))\ne 0$.
\end{definition}

\begin{remark}\label{remark1} Under different mapping $F \in \mathcal{F}$, we obtain the following consequences:
    \begin{enumerate}
        \item [1.] If $F(\xi)=\ln \xi $, \textit{$S^F$-contraction} reduces to \textit{$S^B$-contraction} with $\beta=e^{-\omega  }$.
        \item[2.]  If $F(\xi)=\ln \xi  +\xi$, \textit{$S^F$-contraction} reduces to
        \begin{equation*}
            \frac{\eth(\Upsilon \xi,S(\Upsilon \xi,\Upsilon \zeta))+\eth(\Upsilon \zeta,S(\Upsilon \xi,\Upsilon \zeta))}{\eth(\xi,S(\xi,\zeta))+\eth(\zeta ,S(\xi,\zeta))}\cdot\frac{e^{\eth(\Upsilon \xi,S(\Upsilon \xi,\Upsilon \zeta))+\eth(\Upsilon \zeta,S(\Upsilon \xi,\Upsilon \zeta))}}{e^{\eth(\xi,S(\xi,\zeta))+\eth(\zeta ,S(\xi,\zeta))}}\le e^{-\omega  }.
        \end{equation*}
    \end{enumerate}
\end{remark}

In the following example, we demonstrate that the class of \textit{$S^F$-contraction} contains \textit{$S^B$-contraction}, but not conversely.

\begin{example}
Consider the sequence $(\xi_n)_{n\in\mathbb N}$ defined by
\[
\xi_1=1^3,\quad \xi_2=1^3+2^3,\quad \ldots,\quad 
\xi_n=1^3+2^3+\cdots+n=\left(\frac{n(n+1)}{2}\right)^2.
\]
Let
\[
W=\{\xi_n:\; n\in\mathbb N\}
\]
and equip $W$ with the metric
\[
\eth(\xi,\zeta)=|\xi-\zeta|, \qquad \xi,\zeta \in W.
\]

Define $S(\xi,\zeta)=1$ for $\xi,\zeta \in W$,
and $\Upsilon :W\to W$ by
\[
\Upsilon(\xi_1)=\xi_1, \qquad 
\Upsilon(\xi_n)=\xi_{n-1} \quad \text{for } n>1.
\]
For simplicity, let us denote $O(\xi ,\zeta)=\eth(\xi,S(\xi,\zeta))$ and $P(\xi ,\zeta)=\eth(\zeta ,S(\xi,\zeta))$.
\medskip

\noindent
\textbf{Step 1: $\Upsilon $ is not \textit{$S^B$-contraction}.}

We compute
\begin{align*}
    \frac{O(\Upsilon \xi_n,\Upsilon \xi _1)+P(\Upsilon \xi_n,\Upsilon \xi _1)}{O(\xi_n,\xi_1)+P(\xi_n,\xi_1)}=&\frac{O(\xi_{n-1},\xi_1)+P(\xi_{n-1},\xi_1)}{O(\xi_n,\xi_1)+P(\xi_n,\xi_1)}\\
    =&\frac{\eth(\xi_{n-1},S(\xi_{n-1},\xi_1))+\eth(\xi_{1},S(\xi_{n-1},\xi_1))}{\eth(\xi_{n},S(\xi_{n},\xi_1))+\eth(\xi_{1},S(\xi_{n},\xi_1))} \\
    =& \frac{\eth(\xi_{n-1},1)+\eth(\xi_{1},1)}{\eth(\xi_{n},1)+\eth(\xi_{1},1)}\\
    =& \frac{\xi_{n-1}-1}{\xi_{n}-1}
\end{align*}

Taking limits as $n\to\infty$, we obtain
\[
\lim_{n\to\infty}
\frac{d(\Upsilon \xi_n,S(\Upsilon \xi_n,\Upsilon \xi _1))+d(\Upsilon \xi_1,S(\Upsilon \xi_n,\Upsilon \xi _1))}{d(\xi_n,S(\xi_n,\xi _1))+d(\xi_1,S(\xi_n,\xi _1))}=1.
\]
Therefore, $\Upsilon $ is not an \textit{$S^F$-contraction}.

\medskip

\noindent
\textbf{Step 2: $\Upsilon $ is an \textit{$S^F$-contraction}.}

Observe that for $m,n\in\mathbb N$,
\[
d(\Upsilon \xi_m,S(\Upsilon \xi_m,\Upsilon \xi _n))+d(\Upsilon \xi_n,S(\Upsilon \xi_m,\Upsilon \xi _n))\ne 0
\Longleftrightarrow
\big((m>2 \wedge n=1)\vee (m>n>1)\big).
\]

\medskip

\noindent
\emph{Case 1: $m>2$.}
\begin{align*}
    \frac{O(\Upsilon \xi_m,\Upsilon \xi _1)+P(\Upsilon \xi_m,\Upsilon \xi _1)}{O(\xi_m,\xi_1)+P(\xi_m,\xi_1)}
\cdot\frac{e^{O(\Upsilon \xi_m,\Upsilon \xi _1)+P(\Upsilon \xi_m,\Upsilon \xi _1)}}{e^{O(\xi_m,\xi_1)P(\xi_m,\xi_1)}}
=& \frac{\xi_{m-1}-1}{\xi_m-1}e^{\xi_{m-1}-\xi_m}\\
=& \frac{m^2(m-1)^2-1}{m^2(m+1)^2-1}e^{-m^3}\\
<& e^{-m^3}\\
\le& e^{-27}.
\end{align*}

\medskip

\noindent
\emph{Case 2: $m>n>1$.}
\begin{align*}
    \frac{O(\Upsilon \xi_m,\Upsilon \xi _n)+P(\Upsilon \xi_m,\Upsilon \xi _n)}{O(\xi_m,\xi_n)+P(\xi_m,\xi_n)}
\cdot\frac{e^{O(\Upsilon \xi_m,\Upsilon \xi _n)+P(\Upsilon \xi_m,\Upsilon \xi _n)}}{e^{O(\xi_m,\xi_n)+P(\xi_m,\xi_n)}}
=& \frac{\xi_{m-1}+\xi_{n-1}-2}{\xi_m+\xi_{n}-2}\cdot\frac{e^{\xi_{m-1}+\xi_{n-1}}}{e^{\xi_m+\xi_{n}}}\\
<& \frac{e^{\xi_{m-1}+\xi_{n-1}}}{e^{\xi_m+\xi_{n}}}\\
=& e^{-(n^3+m^3)}\\
\le& e^{-35}.
\end{align*}

Hence, by Remark \ref{remark1}, $\Upsilon $ satisfies the $S^{F}$-contraction with $F(\xi)=\ln \xi +\xi$ 
 and $\omega  =27$.
\end{example}

\begin{lemma}\label{asym lemma}
    Let $(W,\eth,s)$ be a super-metric space and $\Upsilon $ be a \textit{$S^F$-contraction}. Then, $\Upsilon $ is asymptotically regular.
\end{lemma}
\begin{proof}
        Let $\xi_0 \in W$ be arbitrary and define a sequence $\{\xi_n\}$ in $W$ by $\xi_n = \Upsilon \xi _{n-1}, \quad n \in \{1,2,\ldots\}$. Let $\lambda_n = \eth(\xi_{n+1}, S(\xi_{n+1}, \xi_n))$, $ \eta_n = \eth(\xi_{n}, S(\xi_{n+1}, \xi_n)) \text{ for } n \in \{0,1,\ldots\}$.

Now, setting $(\xi ,\zeta)=(\xi_n,\xi_{n-1})$ in \eqref{sf contraction}, we have
\begin{equation}\label{f2}
\begin{split}
    F(\lambda_n+\eta_n)=&F(\eth(\Upsilon \xi_{n},S(\Upsilon \xi_{n},\Upsilon \xi _{n-1}))+\eth(\Upsilon \xi_{n-1},S(\Upsilon \xi_{n},\Upsilon \xi _{n-1})))\\
    \le& F(\eth(\xi_{n},S(\xi_{n},\xi_{n-1}))+\eth(\xi_{n-1},S(\xi_{n},\xi_{n-1})))-\omega  \\
    =& F(\lambda_{n-1}+\eta_{n-1})-\omega  \\
    \le& F(\lambda_{n-2}+\eta_{n-2})-2\omega  \\
    &\vdots\\
    \le& F(\lambda_{0}+\eta_{0})-n\omega  .
    \end{split}
\end{equation}
Letting $n \to \infty$ in inequality \eqref{f2}, we get
\begin{equation}\label{f3}
    \lim_{n\to \infty}F(\lambda_n+\eta_n)=-\infty.
\end{equation}

From (W2) together with \eqref{f3}, we get
\begin{equation}\label{f4}
    \lim_{n \to \infty} \left[\lambda_n+\eta_n \right] =0, \text{ i.e. } \lim_{n \to \infty}\left[\eth(\xi_{n+1}, S(\xi_{n+1}, \xi_n)) + \eth(\xi_{n}, S(\xi_{n+1}, \xi_n))\right] =0.
\end{equation}

    By the definition of limit with \eqref{f4}, we have
    \begin{equation}\label{f5}
        \lim_{n \to \infty}\eth(\xi_{n+1}, S(\xi_{n+1}, \xi_n))=0 \text{ and }  \lim_{n \to \infty}\eth(\xi_{n}, S(\xi_{n+1}, \xi_n))=0.
    \end{equation}
    
  Using super-metric triangle inequality and \eqref{f5}, we have
  $$0 \le \liminf_{n \to \infty}\eth(\xi_n,\xi_{n+1})\le \limsup_{n \to \infty}\eth(\xi_n,\xi_{n+1}) \le s\limsup_{n \to \infty}\eth(S(\xi_{n+1},\xi_n),\xi_{n+1})=0,$$
  \begin{equation}\label{f6}
      \text{ i.e. } \lim_{n \to \infty}\eth(\Upsilon^n\xi_0,\Upsilon^{n+1}\xi_0)=\lim_{n \to \infty}\eth(\xi_n,\xi_{n+1})=0.
  \end{equation}

  Hence, $\Upsilon $ is asymtotically regular.
\end{proof}
\begin{lemma}\label{convergent lemma}
    Let $(W, \eth, s)$ be a complete super-metric space and $\Upsilon : W \to W$ be an asymptotically regular mapping. Then, the Picard iteration ${\Upsilon^n\xi} $ for the initial point $\xi \in W$ is a convergent sequence on $W$. 
\end{lemma}
\begin{proof}
    Since $\Upsilon $ is asymptotically regular mapping, $\lim_{n \to \infty}\eth(\xi_n,\xi_{n+1})=0$.

     Now suppose that $\kappa, n \in \mathbb{N}$ and $\kappa > n$. If $\xi_n = \xi_\kappa$, then
\[
\Upsilon^\kappa(\xi_0) = \Upsilon^n(\xi_0).
\]
So,
\[
\Upsilon^{\kappa-n}(\Upsilon^n(\xi_0)) = \Upsilon^n(\xi_0).
\]
Thus, $\Upsilon^n(\xi_0)$ is a fixed point of $\Upsilon^{\kappa-n}$. Also,
\[
\Upsilon  \big(\Upsilon^{\kappa-n}(\Upsilon^n(\xi_0))\big)
= \Upsilon^{\kappa-n}\big(\Upsilon(\Upsilon^n(\xi_0))\big)
= \Upsilon(\Upsilon^n(\xi_0)).
\]
This means that $\Upsilon(\Upsilon^n(\xi_0))$ is also a fixed point of $\Upsilon^{\kappa-n}$. Hence,
\[
\Upsilon(\Upsilon^n(\xi_0)) = \Upsilon^n(\xi_0).
\]
Therefore, $\Upsilon^n(\xi_0)$ is a fixed point of $\Upsilon $. So, without loss of generality, we may assume that
\[
\xi_n \neq \xi_\kappa \quad \text{for all } \kappa > n.
\]

Now, using the  super-metric triangle inequality, we have
\[
\limsup_{n \to \infty} \eth(\xi_n, \xi_{n+2})
\le s\limsup_{n \to \infty} \eth(\xi_{n+1}, \xi_{n+2}) = 0.
\]

As a result, since $\limsup_{n \to \infty} \eth(\xi_n, \xi_{n+2}) = 0$, we obtain
\[
\limsup_{n \to \infty} \eth(\xi_n, \xi_{n+3})
\le s\limsup_{n \to \infty} \eth(\xi_{n+2}, \xi_{n+3}) = 0.
\]

Now, by induction, we can conclude that
\[
\limsup_{n \to \infty} \{ \eth(\xi_n, \xi_\kappa) : \kappa > n \} = 0.
\]

This implies that the sequence $\{\xi_n\}$ is a Cauchy sequence.  
Since $W$ is a complete super-metric space, $\{\xi_n\}$ is convergent in $W$.
\end{proof}

\begin{theorem}\label{sf theorem}
  Let $(W,\eth,s)$ be a complete super-metric space and $S:W\times W \to W$ be a given mapping. 
If $\Upsilon  : W \to W$ is an \textit{$S^F$-contraction}, then $\Upsilon $ is a \textit{Picard operator} provided that, for all $u,v \in W$,
\[
\lim_{n \to \infty} \eth(\Upsilon^n u, v) = 0 
\quad \Longrightarrow \quad 
\lim_{n \to \infty} \eth(\Upsilon(\Upsilon^n u), \Upsilon  v) = 0.
\]
\end{theorem}
\begin{proof}
 Let $\xi_0 \in W$ be arbitrary and define a sequence $\{\xi_n\}$ in $W$ by $\xi_n = \Upsilon \xi _{n-1}, \quad n \in \{1,2,\ldots\}$. If $\xi_{n_0+1} = \xi_{n_0}$ for some $n_0 \in \{0,1,\ldots\}$, then $\Upsilon \xi_{n_0} = \xi_{n_0}$, and so $\Upsilon $ possesses a fixed point.

Now, suppose that $\xi_{n+1} \neq \xi_n$ for $n \in \{0,1,\ldots\}$.
    Since $\Upsilon : W \to W$ is an \textit{$S^F$-contraction} in super-metric space $(W,\eth,s)$, by Lemma \ref{asym lemma}, $\Upsilon $ is asymptotically regular in $W$. By using Lemma \ref{convergent lemma}, $\Upsilon^n \xi_0$ is convergent, i.e. 
  there exists some $\xi^* \in W$ such that \begin{equation}\label{f7}
    \lim_{n \to \infty} \eth(\Upsilon^n\xi_0,\xi^*)=0,
\end{equation} 
which implies  \begin{equation}\label{f77}
    \lim_{n \to \infty} \eth(\Upsilon(\Upsilon^n\xi_0),\Upsilon \xi ^*)=0.
\end{equation} 
From \eqref{f7} and \eqref{f77}, we can conclude that  $\xi^*$ is the fixed point of $\Upsilon $.

Suppose there exist two fixed points $u,v$ such that $u\ne v$. Now, using \eqref{sf contraction}, we have
\begin{equation*}
    \omega  + F\left(\eth(\Upsilon u,S(\Upsilon u,\Upsilon  v))+\eth(\Upsilon v,S(\Upsilon u,\Upsilon  v))\right) \le F\left(\eth(u,S(u,v))+\eth(v,S(u,v))\right).
\end{equation*}
which is equivalent to
\begin{equation*}
    \omega  + F\left(\eth(u,S(u,v))+\eth(v,S(u,v))\right) \le F\left(\eth(u,S(u,v))+\eth(v,S(u,v))\right).
\end{equation*}
This is a contradiction. 
So, $\Upsilon $ has a unique fixed point, i.e. 
\[
\mathrm{Fix}(F)=\{\xi ^{*}\},
\]
where $\xi^{*}$ is defined by \eqref{f7}. 

Hence, $\Upsilon$ is a Picard operator.
\end{proof}

\begin{example}
Let $W=[0,1]$ and define $\eth:W\times W\to[0,\infty)$ by
\[
\eth(\xi,\zeta)=
\begin{cases}
\xi\zeta, & \xi \neq \zeta,\ \xi,\zeta \in(0,1),\\[4pt]
0, & \xi =\zeta ,\\[4pt]
\zeta, & \xi =0,\ \zeta \in (0,1),\\[4pt]
1-\dfrac{\zeta}{2}, & \xi =1,\ \zeta \in [0,1).
\end{cases}
\]
 Then $(W,\eth,s)$ is a super-metric space.

Define $\Upsilon :W\to W$ by
\[
\Upsilon(\xi)=\frac{\xi }{4}, \quad \xi \in[0,1), 
\qquad 
\Upsilon(1)=\frac{1}{8},
\]
and $S:W\times W\to W$ by
\[
S(\xi,\zeta)=\frac{\xi+\zeta }{2}.
\]

Let $F:(0,\infty)\to\mathbb{R}$ be given by $F(t)=-\frac{1}{\sqrt{t}}$, then $F\in\mathcal{F}$.

For simplicity, denote $L(\xi ,\zeta)= \eth(\Upsilon \xi,S(\Upsilon \xi,\Upsilon \zeta))+\eth(\Upsilon \zeta,S(\Upsilon \xi,\Upsilon \zeta)) $ and $R(\xi ,\zeta)=\eth(\xi,S(\xi,\zeta))+\eth(\zeta ,S(\xi,\zeta))$.

\medskip

\textbf{Case 1:} $\xi,\zeta \in(0,1)$, $\xi \neq \zeta$.

Since $S(\xi,\zeta)=\dfrac{\xi+\zeta }{2}$, $\eth(\xi,S(\xi,\zeta))=\frac{\xi (\xi+\zeta)}{2},
\text{ and }
\eth(\zeta ,S(\xi,\zeta))=\frac{\zeta (\xi+\zeta)}{2}$,
\[
R(\xi ,\zeta)=\eth(\xi,S(\xi,\zeta))+\eth(\zeta ,S(\xi,\zeta))
=
\frac{(\xi+\zeta)^2}{2}.
\]
And, since $S(\Upsilon \xi,\Upsilon \zeta)=\frac{\xi+\zeta }{8}$, $\eth(\Upsilon \xi,S(\Upsilon \xi,\Upsilon \zeta))=\frac{\xi (\xi+\zeta)}{32},
\text{ and }
\eth(\Upsilon \zeta,S(\Upsilon \xi,\Upsilon \zeta))=\frac{\zeta (\xi+\zeta)}{32}$,
\[
L(\xi ,\zeta)=\eth(\Upsilon \xi,S(\Upsilon \xi,\Upsilon \zeta))+\eth(\Upsilon \zeta,S(\Upsilon \xi,\Upsilon \zeta))
=
\frac{(\xi+\zeta)^2}{32}.
\]
Therefore,
\begin{align*}
    F(R(\xi ,\zeta))-F(L(\xi ,\zeta))
=&\frac{4\sqrt{2}}{\xi+\zeta }-\frac{\sqrt{2}}{\xi+\zeta }\\
=&\frac{3\sqrt{2}}{\xi+\zeta }\\
\ge&
\frac{3\sqrt{2}}{2}.
\end{align*}

\medskip
\textbf{Case 2:} $\xi =0$, $\zeta \in (0,1)$.

Since $S(0,\zeta)=\dfrac{\zeta }{2}$, 
$\eth(0,S(0,\zeta))=\dfrac{\zeta }{2},
\text{ and }
\eth(\zeta ,S(0,\zeta))=\dfrac{\zeta^ 2}{2}$,
\[
R(\xi ,\zeta)=\eth(0,S(0,\zeta))+\eth(\zeta ,S(0,\zeta))
=
\frac{\zeta (1+\zeta)}{2}.
\]

And, since $\Upsilon(0)=0$, $\Upsilon(\zeta)=\dfrac{\zeta }{4}$ and $S(\Upsilon 0,\Upsilon \zeta)=\dfrac{\zeta }{8}$,
\[
\eth(\Upsilon \xi,S(\Upsilon \xi,\Upsilon \zeta))=\frac{\zeta }{8},
\text{ and }
\eth(\Upsilon \zeta,S(\Upsilon \xi,\Upsilon \zeta))=\frac{\zeta^ 2}{32},
\]
\[
L(\xi ,\zeta)=\eth(\Upsilon \xi,S(\Upsilon \xi,\Upsilon \zeta))+\eth(\Upsilon \zeta,S(\Upsilon \xi,\Upsilon \zeta))
=
\frac{4\zeta+ \zeta ^2}{32}.
\]

Therefore,
\begin{align*}
F(R(\xi ,\zeta))-F(L(\xi ,\zeta))
=&-\frac{\sqrt{2}}{\sqrt{\zeta (1+\zeta)}}
+\frac{4\sqrt{2}}{\sqrt{\zeta (4+\zeta)}}\\
=&\sqrt{2}
\left(
\frac{4}{\sqrt{\zeta (4+\zeta)}}
-
\frac{1}{\sqrt{\zeta (1+\zeta)}}
\right)\\
>& \frac{4\sqrt{2}}{\sqrt{5}}-1.
\end{align*}

\medskip
\textbf{Case 3:} $\xi =1$, $\zeta \in (0,1)$.

Since $S(1,\zeta)=\dfrac{1+\zeta }{2}$, 
$\eth(1,S(1,\zeta))=\dfrac{3-\zeta}{4},
\text{ and }
\eth(\zeta ,S(1,\zeta))=\dfrac{\zeta (1+\zeta)}{2}$,
\[
R(\xi ,\zeta)=\eth(1,S(1,\zeta))+\eth(\zeta ,S(1,\zeta))
=
\frac{3-\zeta}{4}
+
\frac{\zeta (1+\zeta)}{2}.
\]

And, since $\Upsilon(1)=\dfrac{1}{8}$, $\Upsilon(\zeta)=\dfrac{\zeta }{4}$ and 
$S(\Upsilon 1,\Upsilon \zeta)=\dfrac{1+2\zeta}{16}$,
\[
\eth(\Upsilon \xi,S(\Upsilon \xi,\Upsilon \zeta))=\frac{1+2\zeta}{128},
\text{ and }
\eth(\Upsilon \zeta,S(\Upsilon \xi,\Upsilon \zeta))=\frac{\zeta (1+2\zeta)}{64},
\]
\[
L(\xi ,\zeta)=\eth(\Upsilon \xi,S(\Upsilon \xi,\Upsilon \zeta))+\eth(\Upsilon \zeta,S(\Upsilon \xi,\Upsilon \zeta))
=
\frac{1+2\zeta}{128}
+
\frac{\zeta (1+2\zeta)}{64}.
\]

Therefore, 
\begin{align}
    F(R(\xi ,\zeta))-F(L(\xi ,\zeta))=&\sqrt{\frac{128}{1+4\zeta+ 4\zeta^ 2}}-\sqrt{\frac{4}{3+\zeta +2\zeta^ 2}}\\
    \ge& \frac{8\sqrt{2}}{3}-\frac{2}{\sqrt{6}}
\end{align}

\medskip

Combining all cases, we obtain
\[
\inf_{\xi \neq \zeta}
\left\{
F(R(\xi ,\zeta))-F(L(\xi ,\zeta))
\right\}
=
\frac{4\sqrt{2}}{\sqrt{5}}-1.
\]

Thus, $\Upsilon $ is an \textit{$S^F$-contraction} on $(W,\eth,s)$ with $0<\omega  \le\frac{4\sqrt{2}}{\sqrt{5}}-1$.

For the mapping $\Upsilon :W\to W$ defined above, we have
\[
\Upsilon^n u=
\begin{cases}
\dfrac{u}{4^n}, & u\in[0,1),\\[6pt]
\dfrac{1}{8\cdot 4^{\,n-1}}, & u=1,
\end{cases}
\]
and
\[
\Upsilon^n u \to 0 \quad \text{as } n\to\infty, i.e. \lim_{n\to\infty} \eth(\Upsilon^n u,v)=0..
\]
Also,
\[
\eth(\Upsilon(\Upsilon^n u),\Upsilon  0)=\eth(\Upsilon^{n+1}u,0)\to 0,
\]
which shows that the required condition is satisfied.
\end{example}
\subsection{Corollaries}

    \begin{corollary}\cite{bessem}
   Let $(W,\eth)$ be a complete metric space and $S:W\times W \to W$ be a given mapping. 
If $\Upsilon  : W \to W$ is an \textit{$S^B$-contraction}, then $\Upsilon $ is a \textit{Picard operator} provided that, for all $u,v \in W$,
\[
\lim_{n \to \infty} \eth(\Upsilon^n u, v) = 0 
\quad \Longrightarrow \quad 
\lim_{n \to \infty} \eth(\Upsilon(\Upsilon^n u), \Upsilon  v) = 0.
\]
\end{corollary}
\begin{proof}
    Setting $F(\xi)=\ln \xi $.
\end{proof}
   \begin{corollary}\label{sfcol2}
    Let $(W,\eth,s)$ be a complete super-metric space and $S:W\times W \to W$ be a given mapping. 
If $\Upsilon  : W \to W$ is a continuous \textit{$S^F$-contraction}, then $\Upsilon $ is a \textit{Picard operator}.
\end{corollary}
\begin{proof}
    Since continuity of $\Upsilon $ implies condition “for all $u,v \in W$, $\lim_{n \to \infty} \eth(\Upsilon^n u, v) = 0 
\quad \Longrightarrow \quad 
\lim_{n \to \infty} \eth(\Upsilon(\Upsilon^n u), \Upsilon  v) = 0$"
of Theorem \ref{sf theorem}.
\end{proof}
\begin{corollary}\cite{wardowski}
     Let $(W,\eth)$ be a complete metric space. If $\Upsilon :W\to W$ is an $F$-contraction, then $\Upsilon $ has a unique fixed point.
\end{corollary}
\begin{proof}
    Setting $S(\xi,\zeta)=\xi$ or $S(\xi,\zeta)=\zeta $ in Corollary \ref{sfcol2}.
\end{proof}
\section{\texorpdfstring{\textit{Bianchini $S^F$-contraction}: A generalization of \textit{$S^K$-contraction}}
{Bianchini SF-Contraction: A generalization of SK-Contraction}}

In this section, we introduce the concept of \textit{Kannan $S^F$-contraction} and \textit{Bianchini $S^F$-contraction} and prove that any mapping satisfying \textit{Bianchini $S^F$-contraction} in a complete super-metric space is a \textit{Picard operator} under certain condition.

\begin{definition}\label{ksf definition}
Let $(W,\eth,s)$ be a super-metric space, $F \in \mathcal{F}$, and let 
$S : W \times W \to W$ be a self-mapping. A mapping $\Upsilon  : W \to W$ 
is said to be a \emph{\textit{Kannan $S^F$-contraction}} if the following conditions hold:
\begin{enumerate}
    \item[(i)] 
    If 
    \[
    \eth(\Upsilon \xi,S(\Upsilon \xi,\Upsilon \zeta)) + \eth(\Upsilon \zeta,S(\Upsilon \xi,\Upsilon \zeta)) \neq 0,
    \]
    then
    \[
    \eth(\xi,S(\xi,\Upsilon \xi )) + \eth(\Upsilon \xi,S(\xi,\Upsilon \xi )) \neq 0
    \]
    or,
    \[
    \eth(\zeta ,S(\zeta,\Upsilon \zeta)) + \eth(\Upsilon \zeta,S(\zeta,\Upsilon \zeta)) \neq 0.
    \]

    \item[(ii)] 
    There exists $\omega   > 0$ such that
    \begin{equation}\label{ksf contraction}
    \begin{aligned}
    \omega   + &F(\eth(\Upsilon \xi, S(\Upsilon \xi,\Upsilon \zeta)) + \eth(\Upsilon \zeta, S(\Upsilon \xi,\Upsilon \zeta))) \\
    &\le 
    F\left(
    \frac{
    \eth(\xi,S(\xi,\Upsilon \xi )) + \eth(\Upsilon \xi,S(\xi,\Upsilon \xi )) 
    + \eth(\zeta ,S(\zeta,\Upsilon \zeta)) + \eth(\Upsilon \zeta,S(\zeta,\Upsilon \zeta))
    }{2}
    \right),
    \end{aligned}
    \end{equation}
    for all $\xi,\zeta \in W$.
\end{enumerate}
\end{definition}

\begin{definition}\label{bsf definition}
Let $(W,\eth,s)$ be a super-metric space, $F \in \mathcal{F}$, and let 
$S : W \times W \to W$ be a self-mapping. A mapping $\Upsilon  : W \to W$ 
is said to be a \emph{\textit{Bianchini $S^F$-contraction}} if the following conditions hold:
\begin{enumerate}
    \item[(i)] 
    If 
    \[
    \eth(\Upsilon \xi,S(\Upsilon \xi,\Upsilon \zeta)) + \eth(\Upsilon \zeta,S(\Upsilon \xi,\Upsilon \zeta)) \neq 0,
    \]
    then
    \[
    \eth(\xi,S(\xi,\Upsilon \xi )) + \eth(\Upsilon \xi,S(\xi,\Upsilon \xi )) \neq 0
    \]
    or,
    \[
    \eth(\zeta ,S(\zeta,\Upsilon \zeta)) + \eth(\Upsilon \zeta,S(\zeta,\Upsilon \zeta)) \neq 0.
    \]

    \item[(ii)] 
    There exists $\omega   > 0$ such that
    \begin{equation}\label{bsf contraction}
    \begin{aligned}
    \omega  & + F(\eth(\Upsilon \xi, S(\Upsilon \xi,\Upsilon \zeta)) + \eth(\Upsilon \zeta, S(\Upsilon \xi,\Upsilon \zeta))) \\
    &\le 
    F\left(
    \max\{
    \eth(\xi,S(\xi,\Upsilon \xi )) + \eth(\Upsilon \xi,S(\xi,\Upsilon \xi )) 
    , \eth(\zeta ,S(\zeta,\Upsilon \zeta)) + \eth(\Upsilon \zeta,S(\zeta,\Upsilon \zeta))\}\right),
    \end{aligned}
    \end{equation}
    for all $\xi,\zeta \in W$.
\end{enumerate}
\end{definition}

In the following example, we show that the the class of \textit{Kannan $S^F$-contraction} contains the class of \textit{$S^K$-contraction}, but not conversely.

\begin{example}
Let $W=\{\xi_n : n\in\mathbb{N}\}$ be defined by
\[
\xi_1=1, \qquad \xi_n=3^n,\quad n\ge2.
\]
Let $\eth$ be the Euclidean metric on $W$. Define $S(\xi,\zeta)=\max \{\xi ,\zeta\}, \xi,\zeta \in W$  and $\Upsilon :W\to W$ by
\[
\Upsilon(\xi_n)=
\begin{cases}
\xi_1, & n=1,\\
\xi_{n-1}, & n>1.
\end{cases}
\]

Observe that $\Upsilon \xi=\xi$ if and only if $\xi =\xi_1$ and thus, the condition (i) follows.

Also, for $\xi,\zeta \in W$, $\eth(\Upsilon \xi,S(\Upsilon \xi,\Upsilon \zeta)) + \eth(\Upsilon \zeta,S(\Upsilon \xi,\Upsilon \zeta)) \neq 0$ holds if and only if one of the following occurs:
\[
\xi =\xi_1,\ \zeta =\xi_n\ (n\ge3),
\quad \text{or}\quad
\xi =\xi_2,\ \zeta=\xi_n\ (n\ge3),
\quad \text{or}\quad
\xi =\xi_k,\ \zeta=\xi_r,\ k\ge3,\ r\ne k.
\]

 For simplicity, denote $A(\xi ,\zeta)=\eth(\xi,S(\xi,\Upsilon \xi ))+\eth(\zeta ,S(\zeta,\Upsilon \zeta)$ and $B(\xi ,\zeta)=\eth(\Upsilon \xi,S(\xi,\Upsilon \xi ))+\eth(\Upsilon \zeta,S(\zeta,\Upsilon \zeta)$. We examine three cases.

\medskip
\noindent
\textbf{Case (i)} $\xi =\xi_1$, $\zeta =\xi_n$, $n\ge3$. Then
\begin{equation*}
    \begin{split}
        A(\xi_1,\xi_n)=&\eth(\xi_1,\max \{\xi_1,\Upsilon \xi _1\})+\eth(\xi_n,\max \{\xi_n,\Upsilon \xi _n\})\\
        =& \eth(\xi_1,\xi_1)+\eth(3^n,3^n)\\
        =&0
    \end{split}
\end{equation*}
and 
\begin{equation*}
    \begin{split}
        B(\xi_1,\xi_n)=&\eth(\Upsilon \xi_1,\max \{\xi_1,\Upsilon \xi _1\})+\eth(\Upsilon \xi_n,\max \{\xi_n,\Upsilon \xi _n\})\\
        =&\eth(\xi_1,\xi_1)+\eth(3^{n-1},3^{n})\\
        =&3^n-3^{n-1}
    \end{split}
\end{equation*}
Now, 
\begin{equation}\label{8888}
\begin{split}
    &\frac{\eth(\Upsilon \xi_1,S(\Upsilon \xi_1,\Upsilon \xi _n))+\eth(\Upsilon \xi_n,S(\Upsilon \xi_1,\Upsilon \xi _n))}
     {A(\xi_1,\xi_n)+B(\xi_1,\xi_n)}\\
&= \frac{\eth(1,\max\{1,3^{n-1} \})+\eth(3^{n-1},\max\{1,3^{n-1} \})}{0+3^n-3^{n-1}}\\
&=\frac{3^{\,n-1}-1}{2\cdot 3^{\,n-1}}\\
&<\frac12.
\end{split}
\end{equation}

\medskip
\noindent
\textbf{Case (ii)} $\xi =\xi_2$, $\zeta =\xi_n$, $n\ge3$. Then
\begin{equation*}
    \begin{split}
        A(\xi_2,\xi_n)=&\eth(\xi_2,\max \{\xi_2,\Upsilon \xi _2\})+\eth(\xi_n,\max \{\xi_n,\Upsilon \xi _n\})\\
        =& \eth(\xi_2,\xi_2)+\eth(3^n,3^n)\\
        =&0
    \end{split}
\end{equation*}
and 
\begin{equation*}
    \begin{split}
        B(\xi_2,\xi_n)=&\eth(\Upsilon \xi_2,\max \{\xi_2,\Upsilon \xi _2\})+\eth(\Upsilon \xi_n,\max \{\xi_n,\Upsilon \xi _n\})\\
        =&\eth(\xi_1,\xi_2)+\eth(3^{n-1},3^{n})\\
        =&8+3^n-3^{n-1}.
    \end{split}
\end{equation*}
Now, 

\begin{equation}\label{999}
    \begin{split}
        &\frac{\eth(\Upsilon \xi_2,S(\Upsilon \xi_2,\Upsilon \xi _n))+\eth(\Upsilon \xi_n,S(\Upsilon \xi_2,\Upsilon \xi _n))}
     {A(\xi_2,\xi_n)+B(\xi_2,\xi_n)}\\
&=\frac{\eth(1,\max\{1,3^{n-1} \})+\eth(3^{n-1},\max\{1,3^{n-1} \})}{0+8+3^n-3^{n-1}}\\
&=
\frac{3^{\,n-1}-1}{8+2\cdot 3^{\,n-1}}\\
&<\frac12.
    \end{split}
\end{equation}

\medskip
\noindent
\textbf{Case (iii)} $\xi =\xi_k$, $\zeta =\xi_r$, $k\ge3$, $r\ne k$.

If $r=1$ or $r=2$, then \eqref{8888} and \eqref{999} apply respectively.
Assume $r\notin\{1,2,k\}$. Then
\begin{equation*}
    \begin{split}
        A(\xi_k,\xi_r)=&\eth(\xi_k,\max \{\xi_k,\Upsilon \xi _k\})+\eth(\xi_r,\max \{\xi_r,\Upsilon \xi_r\})\\
        =& \eth(\xi_k,\xi_k)+\eth(3^r,3^r)\\
        =&0
    \end{split}
\end{equation*}
and 
\begin{equation*}
    \begin{split}
        B(\xi_k,\xi_r)=&\eth(\Upsilon \xi_k,\max \{\xi_k,\Upsilon \xi _k\})+\eth(\Upsilon \xi_r,\max \{\xi_r,\Upsilon \xi _r\})\\
        =&\eth(\xi_{k-1},\xi_k)+\eth(3^{r-1},3^{r})\\
        =&3^k-3^{k-1}+3^r-3^{r-1}\\
        =& 2(3^{k-1}+3^{r-1}).
    \end{split}
\end{equation*}
Now, 

\begin{equation}\label{101010}
    \begin{split}
        &\frac{\eth(\Upsilon \xi_k,S(\Upsilon \xi_k,\Upsilon \xi_r))+\eth(\Upsilon \xi_r,S(\Upsilon \xi_k,\Upsilon \xi _r))}
     {A(\xi_k,\xi_r)+B(\xi_k,\xi_r)}\\
&= \frac{\eth(3^{k-1},\max\{3^{k-1},3^{r-1} \})+\eth(3^{r-1},\max\{3^{k-1},3^{r-1} \})}{0+2(3^{k-1}+3^{r-1})}\\
&=
\frac{|3^{k-1}-3^{r-1}|}
     {2(3^{k-1}+3^{r-1})}\\
&<\frac12.
    \end{split}
\end{equation}

From \eqref{8888}, \eqref{999}, and \eqref{101010}, we obtain
\[\frac{\eth(\Upsilon \xi,S(\Upsilon \xi,\Upsilon \zeta))+\eth(\Upsilon \zeta,S(\Upsilon \xi,\Upsilon \zeta))}{\eth(\xi,S(\xi,\Upsilon \xi ))+\eth(\Upsilon \xi,S(\xi,\Upsilon \xi ))+\eth(\zeta ,S(\zeta,\Upsilon \zeta))+\eth(\Upsilon \zeta,S(\zeta,\Upsilon \zeta))} < \frac12\]

for all $\xi,\zeta \in W$ with $\Upsilon \xi\ne \Upsilon \zeta$.

Moreover, if $\xi =\xi_1$ and $\zeta =\xi_n$, then
\[
\frac{\eth(\Upsilon \xi_1,S(\Upsilon \xi_1,\Upsilon \xi _n))+\eth(\Upsilon \xi_n,S(\Upsilon \xi_1,\Upsilon \xi_n))}
     {A(\xi_1,\xi_n)+B(\xi_1,\xi_n)}
\to \frac12
\quad \text{as } n\to\infty.
\]
So, $\Upsilon $ is not \textit{$S^K$-contraction}.

\medskip
Finally, for each $\xi,\zeta \in W$ with $\Upsilon \xi\ne \Upsilon \zeta$, consider
\begin{align*}
\frac{\eth(\Upsilon \xi,S(\Upsilon \xi,\Upsilon \zeta))+\eth(\Upsilon \zeta,S(\Upsilon \xi,\Upsilon \zeta))}{\frac{A(\xi ,\zeta)+B(\xi ,\zeta)}{2}}
\cdot&\frac{e^{\left(
\eth(\Upsilon \xi,S(\Upsilon \xi,\Upsilon \zeta))+\eth(\Upsilon \zeta,S(\Upsilon \xi,\Upsilon \zeta))
\right)}}{e^{\frac{A(\xi ,\zeta)+B(\xi ,\zeta)}{2}}}\\
<&\frac{e^{\left(
\eth(\Upsilon \xi,S(\Upsilon \xi,\Upsilon \zeta))+\eth(\Upsilon \zeta,S(\Upsilon \xi,\Upsilon \zeta))
\right)}}{e^{\frac{A(\xi ,\zeta)+B(\xi ,\zeta)}{2}}}\\
=&\begin{cases}
\frac{1}{e}, & \text{by \eqref{8888}},\\
\frac{1}{e^5}, & \text{by \eqref{999}},\\
\frac{1}{e^{2\cdot 3^{r-1}}} \ \text{or}\ \frac{1}{e^{2\cdot 3^{k-1}}}, & \text{by \eqref{101010}}
\end{cases}\\
<& e^{-1/2}.
\end{align*}

Therefore, $\Upsilon $ is \textit{Kannan $S^F$-contraction} with $\omega  =1/2$.
\end{example}

In the following example, we show that the the class of \textit{Bianchini $S^F$-contraction} contains the class of \textit{Kannan $S^F$-contraction}, but not conversely.
\begin{example}
    Let $W=[0,1]$ and $\eth$ be the Euclidean metric on $W$. Define $S(\xi,\zeta)=\min \{\xi ,\zeta\}, \xi,\zeta \in W$  and $\Upsilon :W\to W$ by
\[
\Upsilon(\xi)=
\begin{cases}
0.95, & \xi \in [0,1),\\
0.7, & \xi =1.
\end{cases}
\]
For simplicity, denote $C(\xi ,\zeta)=\eth(\xi,S(\xi,\Upsilon \xi ))+\eth(\Upsilon \xi,S(\xi,\Upsilon \xi ))$ and $D(\xi ,\zeta)=\eth(\zeta ,S(\zeta,\Upsilon \zeta)+\eth(\Upsilon \zeta,S(\zeta,\Upsilon \zeta)$.

Observe that $\Upsilon \xi=\xi$ if and only if $\xi =0.95$ and thus, the condition (i) follows.

First we show that $\Upsilon $ is not \textit{Kannan $S^F$-contraction}. Indeed, taking $\xi =0.8$ and $\zeta =1$
\begin{align*}
    A(0.8,1)=&\eth(0.8,S(0.8,\Upsilon  0.8))+\eth(1,\eth(1,\Upsilon  1))\\
    =& \eth(0.8,S(0.8,0.95))+\eth(1,\eth(1,0.7))\\
    =& \eth(0.8,0.8)+\eth(1,0.7)\\
    =& 0.3
\end{align*}
and 
\begin{align*}
    B(0.8,1)=&\eth(\Upsilon 0.8,S(0.8,\Upsilon  0.8))+\eth(\Upsilon 1,\eth(1,\Upsilon  1))\\
    =& \eth(0.95,S(0.8,0.95))+\eth(0.7,\eth(1,0.7))\\
    =& \eth(0.95,0.8)+\eth(0.7,0.7)\\
    =& 0.15.
\end{align*}
Then,
\begin{align*}
    F(\eth(\Upsilon 0.8,S(\Upsilon 0.8,\Upsilon  1))+\eth(\Upsilon 1,S(\Upsilon 0.8,\Upsilon  1)))=& F(\eth(0.95,S(0.95,0.7))+\eth(0.7,S(0.95,0.7)))\\
    =& F(\eth(0.95,0.7)+\eth(0.7,0.7))\\
    =& F(0.25)\\
    >& F(0.225)\\
    =& F\left(\frac{0.3+0.15}{2}\right)\\
    =& F\left(\frac{A(0.8 ,1 )+B(0.8 ,1 )}{2}\right)
\end{align*}
So, $\Upsilon $ is not \textit{Kannan $S^F$-contraction} since we cannot find any $\omega  >0$ satisfying \eqref{ksf contraction} for $\xi =0.8$ and $\zeta =1$.

On the other hand, for $\xi \in [0,1)$ and $\zeta =1$
\begin{align*}
    C(\xi ,\zeta)=&\eth(\xi,S(\xi,\Upsilon \xi ))+\eth(\Upsilon \xi,S(\xi,\Upsilon \xi ))\\
    =&\eth(\xi,S(\xi,0.95))+\eth(0.95,S(\xi,0.95))\\
    =&\eth(\xi,\min\{\xi ,0.95\})+\eth(0.95,\min\{\xi ,0.95\})\\
    =& \begin{cases}
       |\xi -0.95|\ge 0, \text{ when } 0.95\le \xi< 1\\
       |\xi -0.95|> 0, \text{ when } 0.65\le \xi < 0.95\\
        |\xi -0.95|\ge 0.3, \text{ when } 0\le \xi < 0.65
    \end{cases}\\
\ge& \min\{0,0,0.3\}\\
=& 0,
\end{align*}
and
\begin{align*}
    D(\xi ,\zeta)=&\eth(1,S(1,\Upsilon  1))+\eth(\Upsilon 1,S(1,\Upsilon  1))\\
    =&\eth(1,S(1,0.7))+\eth(0.7,S(1,0.7))\\
    =&\eth(1,0.7)+\eth(0.7,0.7)\\
    =&0.3.
\end{align*}
Then,
\begin{align*}
F\left(\eth(\Upsilon \xi,S(\Upsilon \xi,\Upsilon \zeta))+\eth(\Upsilon \zeta,S(\Upsilon \xi,\Upsilon \zeta))\right)=&F\left(\eth(0.95,S(0.95,0.7))+\eth(0.7,S(0.95,0.7))\right)\\
=&F\left(\eth(0.95,0.7)+\eth(0.7,0.7)\right)\\
=&F(0.25).
\end{align*}
Now, 
 \begin{align*}
   F\left(\max\{C(\xi ,\zeta),D(\xi ,\zeta)\}\right)- &F\left(\eth(\Upsilon \xi,S(\Upsilon \xi,\Upsilon \zeta))+\eth(\Upsilon \zeta,S(\Upsilon \xi,\Upsilon \zeta))\right)\\
   \ge& F\left(\max\{0,0.3\}\right)-F\left(0.25\right)\\
   =& F(0.3)-F(0.25).
 \end{align*}
\end{example}
 So, $\Upsilon $ is \textit{Bianchini $S^F$-contraction} with $\omega  =F(0.3)-F(0.25)>0$

\begin{lemma}\label{asym lemma2}
    Let $(W,\eth,s)$ be a super-metric space and $\Upsilon $ be a \textit{Bianchini $S^F$-contraction}. Then, $\Upsilon $ is asymptotically regular.
\end{lemma}
\begin{proof}
    Let $\xi_0 \in W$ be arbitrary and define a sequence $\{\xi_n\}$ in $W$ by $\xi_n = \Upsilon \xi_{n-1}, \quad n \in \{1,2,\ldots\}$.

Now suppose that $\xi_{n+1} \neq \xi_n$ for $n \in \{0,1,\ldots\}$ and let $\lambda_{n+1} = \eth(\xi_{n+1}, S(\xi_{n+1}, \xi_{n+2}))$, $ \eta_{n+1} = \eth(\xi_{n+2}, S(\xi_{n+1}, \xi_{n+2})) \text{ for } n \in \{0,1,\ldots\}$.

Then $\lambda_n > 0$ for all $n \in \{0,1,\ldots\}$.

Now, setting $(\xi ,\zeta)=(\xi_n,\xi_{n+1})$ in \eqref{bsf contraction}, we have
\begin{equation}\label{bf2}
\begin{split}
    F(\lambda_{n+1}+\eta_{n+1})=&F(\eth(\Upsilon \xi_{n},S(\Upsilon \xi_{n},\Upsilon \xi _{n+1}))+\eth(\Upsilon \xi_{n+1},S(\Upsilon \xi_{n},\Upsilon \xi_{n+1})))\\
    \le& F(\max \left\{C(\xi_n,\xi_{n+1}), D(\xi_n,\xi_{n+1})\right\})-\omega  \\
    =& F(\max\{\lambda_{n}+\eta_{n}, \lambda_{n+1}+\eta_{n+1}\})-\omega  .
    \end{split}
\end{equation}
Suppose $\{\lambda_{n}+\eta_{n}\}$ is an increasing sequence in $W$, then inequality \eqref{bf2} becomes $0\le-\omega  $. This is a contradiction. Hence, $\{\lambda_{n}+\eta_{n}\}$ is a strictly decreasing sequence.

Now, from \eqref{bf2} \begin{equation}
    \begin{split}
        F(\lambda_{n+1}+\eta_{n+1})\le& F(\lambda_{n}+\eta_{n})-\omega  \\
        \le& F(\lambda_{n-1}+\eta_{n-1})-2\omega  \\
    &\vdots\\
    \le& F(\lambda_{0}+\eta_{0})-n\omega  .
    \end{split}
\end{equation}
Following the proof of Lemma \ref{asym lemma}, we can prove $\lim_{n \to \infty}\eth(\xi_n,\xi_{n+1})=0$. Hence, $\Upsilon $ is asymptotically regular.
\end{proof}

\begin{figure}[h!]
\centering
\begin{tikzpicture}[
    scale=0.9, 
    transform shape,
    >=latex,
    box/.style={
        draw,
        rectangle,
        align=center,
        minimum width=3cm,
        minimum height=1cm
    }
]

\node[box] (bianchini) at (0,5.7)
    {Bianchini\\$s^{F}$-contraction};

\node[box] (kannans) at (0,3.5)
    {Kannan\\$s^{F}$-contraction};

\node[box] (kannant) at (-4,1.8)
    {Kannan\\$F$-contraction};

\node[box] (sk) at (4,1.8)
    {$S^{K}$-contraction};

\node[box] (kannan) at (0,0)
    {Kannan\\contraction};

\draw[->, thick] (kannan) -- (kannant);
\draw[->, thick] (kannan) -- (sk);

\draw[->, thick] (kannant) -- (kannans);
\draw[->, thick] (sk) -- (kannans);

\draw[->, thick] (kannans) -- (bianchini);

\end{tikzpicture}
\caption{Hierarchy of \textit{Kannan contraction} and its generalizations}
\label{fig:contraction-hierarchy}
\end{figure}
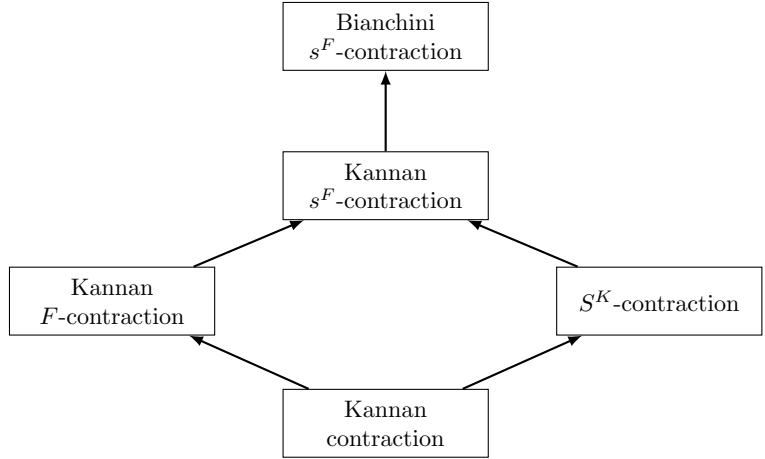

\begin{theorem}\label{bsf theorem}
      Let $(W,\eth,s)$ be a complete super-metric space and $S:W\times W \to W$ be a given mapping  with $S(\xi,\xi)=\xi$. 
If $\Upsilon  : W \to W$ is an \textit{Bianchini $S^F$-contraction}, then $\Upsilon $ is a \textit{Picard operator} provided that, for all $u,v \in W$,
\[
\lim_{n \to \infty} \eth(\Upsilon^n u, v) = 0 
\quad \Longrightarrow \quad 
\lim_{n \to \infty} \eth(\Upsilon(\Upsilon^n u), \Upsilon  v) = 0.
\]
\end{theorem}
\begin{proof}
     Let $\xi_0 \in W$ be arbitrary and define a sequence $\{\xi_n\}$ in $W$ by $\xi_n = \Upsilon \xi _{n-1}, \quad n \in \{1,2,\ldots\}$.

If $\xi_{n_0+1} = \xi_{n_0}$ for some $n_0 \in \{0,1,\ldots\}$, then $\Upsilon \xi_{n_0} = \xi_{n_0}$, and so $\Upsilon $ possesses a fixed point.

Now suppose that $\xi_{n+1} \neq \xi_n$ for $n \in \{0,1,\ldots\}$
    Since $\Upsilon : W \to W$ is an \textit{$S^F$-contraction} in super-metric space $(W,\eth,s)$, by Lemma \ref{asym lemma2}, $\Upsilon $ is asymptotically regular in $W$. By using Lemma \ref{convergent lemma}, $\Upsilon^n \xi_0$ is convergent, i.e. 
  there exists some $\xi^* \in W$ such that \begin{equation}\label{bf7}
    \lim_{n \to \infty} \eth(\Upsilon^n\xi_0,\xi^*)=0,
\end{equation} 
which implies  \begin{equation}\label{bf77}
    \lim_{n \to \infty} \eth(\Upsilon(\Upsilon^n\xi_0),\Upsilon \xi ^*)=0.
\end{equation} 
From \eqref{bf7} and \eqref{bf77}, we can conclude that  $\xi^*$ is the fixed point of $\Upsilon $.

Suppose there exist two fixed points $u,v$ such that $u\ne v$. Now using \eqref{bsf contraction}  with $S(\xi,\xi)=\xi$, we have
\begin{align*}
   F\left(\eth(u,S(u,v))+\eth(v,S(u,v))\right)<& \omega  + F\left(\eth(u,S(u,v))+\eth(v,S(u,v))\right)\\
   \le& F\left(\max\{\eth(u,u)+\eth(u,u),\eth(v,v)+\eth(v,v)\}\right)\\
   =& F(0).
\end{align*}
Since $F$ is strictly increasing, the above inequality implies that 
\begin{equation*}
    \eth(u,S(u,v))+\eth(v,S(u,v))=0,
\end{equation*}
i.e. \begin{equation}\label{bf8}
    u=S(u,v)=v.
\end{equation} 
So, $\Upsilon $ has a unique fixed point.
\end{proof}
\begin{example}
Let $W=\{1,2,3,4\}$ and let $m:W\times W\to [0,+\infty)$ be defined by
\[
\eth(\xi ,\zeta)=\eth(\zeta,\xi)=|\xi-\zeta|^2,\quad \text{for } \xi,\zeta \in\{2,3,4\},
\]
and
\[
\eth(1,\xi)=\eth(\xi ,1)=(1-\xi^3)^2,\quad \text{for } \xi \in W.
\]
Then $(W,\eth)$ is a super-metric space with coefficient $s=2$.

Define $\Upsilon :W\to W$ by
\[
\Upsilon  1=\Upsilon  4=2,\qquad \Upsilon  2=\Upsilon  3=3,
\]
and define $S:W\times W\to W$ by
\[
S(\xi,\zeta)=
\begin{cases}
1, & \xi \ne \zeta,\\
\xi, & \xi =\zeta .
\end{cases}
\]

For simplicity, denote
\[
C(\xi ,\zeta)=\eth(\xi,S(\xi,\Upsilon \xi ))+\eth(\Upsilon \xi,S(\xi,\Upsilon \xi ))
\]
and
\[
D(\xi ,\zeta)=\eth(\zeta ,S(\zeta,\Upsilon \zeta))+\eth(\Upsilon \zeta,S(\zeta,\Upsilon \zeta)).
\]

Also, for $\xi,\zeta \in W$, $\eth(\Upsilon \xi,S(\Upsilon \xi,\Upsilon \zeta)) + \eth(\Upsilon \zeta,S(\Upsilon \xi,\Upsilon \zeta)) \neq 0$ holds if and only if the following pairs occur:
$(1,2),(1,3),(4,2),(4,3)$.

\medskip

Consider $\xi =4$ and $\zeta =3$.

\begin{align*}
C(4,3)=&\,\eth(4,S(4,\Upsilon  4))+\eth(\Upsilon 4,S(4,\Upsilon  4))\\
=&\,\eth(4,S(4,2))+\eth(2,S(4,2))\\
=&\,\eth(4,1)+\eth(2,1)\\
=&\,(1-4^3)^2+(1-2^3)^2\\
=&\,63^2+7^2\\
=&\,3969+49\\
=&\,4018,
\end{align*}

\begin{align*}
D(4,3)=&\,\eth(3,S(3,\Upsilon  3))+\eth(\Upsilon 3,S(3,\Upsilon  3))\\
=&\,\eth(3,S(3,3))+\eth(3,S(3,3))\\
=&\,\eth(3,3)+\eth(3,3)\\
=&\,0.
\end{align*}

Hence,
\[
\max\{C(4,3),D(4,3)\}=4018.
\]

\medskip

Next,
\begin{align*}
&\eth(\Upsilon 4,S(\Upsilon 4,\Upsilon  3))+\eth(\Upsilon 3,S(\Upsilon 4,\Upsilon  3))\\
=&\,\eth(2,S(2,3))+\eth(3,S(2,3))\\
=&\,\eth(2,1)+\eth(3,1)\\
=&\,49+676\\
=&\,725.
\end{align*}

Therefore,
\begin{align*}
F\left(\max\{C(4,3),D(4,3)\}\right)
- F\left(\eth(\Upsilon 4,S(\Upsilon 4,\Upsilon  3))+\eth(\Upsilon 3,S(\Upsilon 4,\Upsilon  3))\right)
=&\,F(4018)-F(725).
\end{align*}

Since $4018>725$ and $F\in\mathcal F$ is strictly increasing,
\[
F(4018)-F(725)>0.
\]

Thus, for
\[
\omega  =F(4018)-F(725)>0,
\]
the inequality
\[
\omega  +F\left(\eth(\Upsilon \xi,S(\Upsilon \xi,\Upsilon \zeta))+\eth(\Upsilon \zeta,S(\Upsilon \xi,\Upsilon \zeta))\right)
\le
F\left(\max\{C(\xi ,\zeta),D(\xi ,\zeta)\}\right).
\]
One can easily prove for the remaining pairs. Hence, $\Upsilon $ is a \textit{Bianchini $S^F$-contraction} on $(W,\eth,s)$.

Also, $\Upsilon^n u=3$ for all $u\in W$, i.e. $\lim_{n\to\infty}\eth(\Upsilon^n u,3)=0$.

And $\lim_{n\to\infty}\eth(\Upsilon(\Upsilon^n u),\Upsilon  3)=\lim_{n\to\infty}\eth(\Upsilon^{n+1} u,3)=0$
\end{example}
 
\subsection{Corollaries}
 \begin{corollary}\label{bsfcol11}
         Let $(W,\eth)$ be a complete metric space and $S:W\times W \to W$ be a given mapping  with $S(\xi,\xi)=\xi$. 
If $\Upsilon  : W \to W$ is an \textit{Kannan $S^F$-contraction}, then $\Upsilon $ is a \textit{Picard operator} provided that, for all $u,v \in W$,
\[
\lim_{n \to \infty} \eth(\Upsilon^n u, v) = 0 
\quad \Longrightarrow \quad 
\lim_{n \to \infty} \eth(\Upsilon(\Upsilon^n u), \Upsilon  v) = 0.
\]
\end{corollary}
\begin{proof}
    Using the inequality $\frac{a+b}{2}\le \max\{a,b\}$ in Theorem \ref{bsf theorem}.
\end{proof}

    \begin{corollary}\label{bsfcol1}\cite{bessem}
         Let $(W,\eth)$ be a complete metric space and $S:W\times W \to W$ be a given mapping  with $S(\xi,\xi)=\xi$. 
If $\Upsilon  : W \to W$ is an \textit{$S^K$-contraction}, then $\Upsilon $ is a \textit{Picard operator} provided that, for all $u,v \in W$,
\[
\lim_{n \to \infty} \eth(\Upsilon^n u, v) = 0 
\quad \Longrightarrow \quad 
\lim_{n \to \infty} \eth(\Upsilon(\Upsilon^n u), \Upsilon  v) = 0.
\]
\end{corollary}
\begin{proof}
     Setting $F(\xi)=\ln \xi $ in corollary \ref{bsfcol11}.
\end{proof}
   \begin{corollary}\label{bsfcol2}
         Let $(W,\eth,s)$ be a complete super-metric space and $S:W\times W \to W$ be a given mapping  with $S(\xi,\xi)=\xi$. 
If $\Upsilon  : W \to W$ is a continuous \textit{Bianchini $S^F$-contraction}, then $\Upsilon $ is a \textit{Picard operator}.
\end{corollary}
\begin{proof}
    Since continuity of $\Upsilon $ implies condition “ for all $u,v \in W$, $\lim_{n \to \infty} \eth(\Upsilon^n u, v) = 0 
\quad \Longrightarrow \quad 
\lim_{n \to \infty} \eth(\Upsilon(\Upsilon^n u), \Upsilon  v) = 0$
" of Theorem \ref{bsf theorem}.
\end{proof}
\begin{corollary}\cite{bianchini}
     Let $(W,\eth)$ be a complete metric space. If $\Upsilon :W\to W$ is a Bianchini-contraction, then $\Upsilon $ has a unique fixed point.
\end{corollary}
\begin{proof}
    Setting $S(\xi,\zeta)=\xi$ or $S(\xi,\zeta)=\zeta $ and $F(\xi)=\ln \xi $ in Corollary \ref{bsfcol2}.
\end{proof}
\begin{corollary}\cite{batra}
     Let $(W,\eth)$ be a complete metric space. If $\Upsilon :W\to W$ is a \textit{Kannan $F$-contraction}, then $\Upsilon $ has a unique fixed point.
\end{corollary}
\begin{proof}
    Setting $S(\xi,\zeta)=\xi$ or $S(\xi,\zeta)=\zeta $ in Corollary \ref{bsfcol11}.
\end{proof}


\section{Application}
 Inspired by \cite{ieee}, we apply our main result to aircraft terrain following.

Consider an aircraft flying over a terrain profile. For modeling purposes, we restrict ourselves to a two–dimensional framework involving the horizontal distance and the altitude; thus the motion is described in the pitch plane of the aircraft. The terrain profile is detected using radar measurements, while the aircraft is equipped with a radar altimeter to measure its height above the ground. 

In practice, the aircraft is subject to several operational constraints. In particular, limitations on acceleration can be interpreted as restrictions on the radius of curvature of the flight trajectory when the speed is fixed. Moreover, there may also exist an upper bound on the inclination angle of the flight relative to the horizontal axis. Taking these constraints together with a prescribed ground clearance, one can determine a desired flight trajectory denoted by $\gamma(\xi)$, where $\xi$ represents the horizontal distance.

The problem is to generate automatically a flight path that satisfies the above requirements. To describe the system, we introduce the following functions:

\begin{itemize}
\item $\gamma(\xi)$: the desired trajectory specifying the required altitude above the terrain at each horizontal position $\xi$;
\item $\varkappa(\xi)$: the control input associated with the flight control system, for instance the elevator control signal;
\item $\sigma(\xi)$: the trajectory determined by the control function $\varkappa(\xi)$, which may represent the flight angle along the path;
\item $\varpi(\xi)$: the actual altitude of the aircraft during flight.
\end{itemize}

The control function $\varkappa(\xi)$ is updated iteratively through a feedback mechanism. Starting from an initial control $\varkappa_0(\xi)$, the sequence $\{\varkappa_n(\xi)\}$ is generated so that the resulting trajectory approaches the required path $\gamma(\xi)$. The updating scheme is given by
\[
\varkappa_{n+1}(\xi)=\mu_n(\xi)
= -\gamma(\xi)+\varkappa_n(\xi)+\varpi_n(\xi),
\]
where $\varpi_n(\xi)$ denotes the altitude obtained at the $n$th iteration.

The altitude $\varpi_n(\xi)$ depends on the control input $\varkappa_n(\xi)$ according to
\begin{equation}\label{aa1}
\varpi_n(\xi)=\mathcal{G}K\,\varkappa_n(\xi).
\end{equation}
Substituting this relation into the recursive formula yields
\begin{equation}\label{aa2}
\varkappa_{n+1}(\xi)
= -\gamma(\xi)+(\mathcal{G}K+I)\varkappa_n(\xi),
\end{equation}
where $\mathcal{G}$ is a scaling parameter and $K$ represents the dynamic characteristics of the aircraft.

For the convergence of the iterative process, it is sufficient that the operator on the right–hand side of the above equation defines an $S^F$-contraction. Under this assumption, the iteration converges to a stable control function for which the actual trajectory coincides with the prescribed one.

Consequently, the solution of the system satisfies
\[
\varkappa(\xi)= -\gamma(\xi)+(\mathcal{G}K+I)\varkappa(\xi).
\]
Rearranging the above relation, we obtain
\[
\gamma(\xi)=\mathcal{G}K\,\varkappa(\xi)=\varpi(\xi).
\]
Hence, the obtained control $\varkappa(\xi)$ ensures that the realized altitude $\varpi(\xi)$ coincides with the required trajectory $\gamma(\xi)$.

Assume now that the aircraft moves with a constant speed $S$. Then the altitude $\varpi(\xi)$ relative to a reference level can be written as
\[
\varpi(\xi)=\varpi(\xi_0)+S\int_{\xi_0}^{\xi}\tan(\sigma)\,d\xi,
\]
where $\sigma(\xi)$ denotes the flight angle and $\xi_0$ is the initial horizontal position. For simplicity, we assume that
\begin{equation}\label{aa3}
\varpi(\xi)=\mathcal{G}\sigma(\xi).
\end{equation}
This relation enables the description of the aircraft dynamics in terms of the flight angle $\sigma(\xi)$.

The dynamics of the aircraft are governed by the following third–order differential equation:
\begin{equation}\label{aa4}
\frac{d^3\sigma}{dt^3}
+a_2\frac{d^2\sigma}{dt^2}
+a_1\frac{d\sigma}{dt}
=
b_2\frac{d^2\varkappa}{dt^2}
+b_1\frac{d\varkappa}{dt}
-b_0\varkappa.
\end{equation}

Using Laplace transform techniques, the solution can be expressed in the form
\[
f[\xi(t)]
=
b_2\frac{d^2\varkappa}{dt^2}
+
b_1\frac{d\varkappa}{dt}
-
b_0\varkappa.
\]
Since the horizontal distance $\xi$ depends on time $t$ for a given speed, both $\sigma$ and $\varkappa$ may be regarded as functions of either $t$ or $\xi$.

From Eq.~\eqref{aa4}, the solution for $\sigma(\xi)$ can be written as
\[
\sigma[\xi(t)]
=
\sigma(\xi_0)
+
\int_{t_0}^{t} e^{\beta t_1}
\int_{t_0}^{t_1} e^{-\beta t_2}e^{\alpha t_2}
\int_{t_0}^{t_3} e^{-\alpha t_3}
f[\xi(t_3)]\,dt_3\,dt_2\,dt_1,
\]
which ultimately reduces to
\[
\sigma(\xi(t))=K\,\varkappa(\xi),
\]
where $K$ and $\mathcal{G}$ are constants determined by the dynamic properties of the system.

The block diagram of the system is given in Fig. \ref{block diagram}
\begin{figure}[h!]
\centering
\begin{tikzpicture}[>=latex, thick, scale=0.66, transform shape]

\tikzstyle{sum}=[draw, circle, minimum size=5mm, fill=orange!50]
\tikzstyle{sum2}=[draw, circle, minimum size=5mm, fill=green!40]
\tikzstyle{blockC}=[draw, rectangle, minimum width=1.3cm, minimum height=0.9cm, fill=blue!40]
\tikzstyle{blockA}=[draw, rectangle, minimum width=1.3cm, minimum height=0.9cm, fill=red!50]
\tikzstyle{blockH}=[draw, rectangle, minimum width=1.3cm, minimum height=0.9cm, fill=purple!60]

\node (input) at (0,0) {$\Upsilon(\xi)$};
\node[sum] (sum1) at (2.2,0) {};
\node[blockC] (C) at (4,0) {$C$};
\node[blockA] (K) at (9.9,0) {Airplane dynamics $K$};
\node[blockH] (H) at (14.6,0) {$\mathcal{G}$};
\node (output) at (17.5,0) {actual position $\varpi(\xi)$};

\node[sum2] (sum2node) at (9.8,-3) {};

\draw ($(sum1.center)+(-5pt,6pt)$) -- ($(sum1.center)+(5pt,-6pt)$);
\draw ($(sum1.center)+(-5pt,-6pt)$) -- ($(sum1.center)+(5pt,6pt)$);

\draw ($(sum2node.center)+(-5pt,6pt)$) -- ($(sum2node.center)+(5pt,-6pt)$);
\draw ($(sum2node.center)+(-5pt,-6pt)$) -- ($(sum2node.center)+(5pt,6pt)$);

\node at ([xshift=-4pt,yshift=4pt]sum1.north west) {\small $+$};
\node at ([xshift=-4pt,yshift=-4pt]sum1.south west) {\small $-$};

\node at ([xshift=-4pt,yshift=4pt]sum2node.north west) {\small $+$};
\node at ([xshift=-4pt,yshift=-4pt]sum2node.south west) {\small $-$};

\draw[->] (input) -- (sum1);
\draw[->] (sum1) -- node[above] {$\mu(\xi)$} (C);
\draw[->] (C) -- node[above] {control function $\varkappa(\xi)$} (K);
\draw[->] (K) -- node[above] {output $\gamma(\xi)$} (H);
\draw[->] (H) -- (output);

\draw[->] (output) |- (sum2node);
\draw[->] (sum2node) -| node[left] {$\varpi(\xi)+\varkappa(\xi)$} (sum1);

\end{tikzpicture}

\caption{Block diagram of the automatic flight-path generator}
\label{block diagram}
\end{figure}
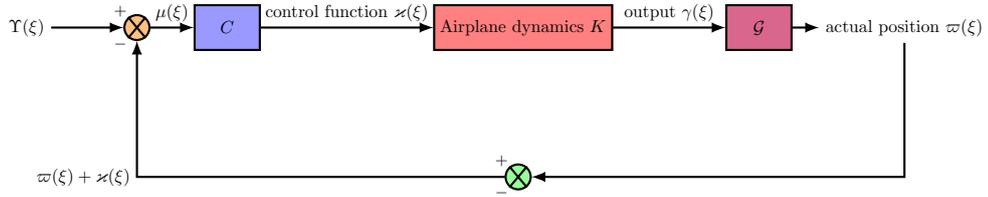
\begin{theorem}\label{appth}
     Suppose that there exists $\omega  ,\delta,a>0$ and $p\ge 1$ and a sequence $\Delta_n=|\varkappa_{n}-\varkappa_{n-1}-a|^p$ such that the conditions given below are satisfied.

\begin{itemize}
\item[(F$_1$)] $(\xi - \xi_0)S\delta \le 1 + e^{\frac{1}{p}(\Delta_n-\Delta_{n+1}-\omega  )},$
\item[(F$_2$)] $\dfrac{\mathcal{G}K\varkappa_{n}(\xi) - \mathcal{G}K\varkappa_{n-1}(\xi)}{\varkappa_n - \varkappa_{n-1}-a} < -1 + e^{\frac{1}{p}(\Delta_n-\Delta_{n+1}-\omega  )}.$
\end{itemize}

Then, \eqref{aa2} converges.
\end{theorem}
\begin{proof}
    
 Let us define the metric $\eth(f(\xi), g(\xi)) = \max |f(\xi) - g(\xi)|^p,p\ge1$, which is a complete super-metric space with $s=2^{p-1}$.

If the operator on the right-hand side of \eqref{aa2} is an $S^F$-contraction, then it possesses a unique fixed point $\xi^*$ given by
\[
\xi^* = -\gamma + (\mathcal{G}K + I)\xi^*.
\]

Taking $S(\xi,\zeta)=\zeta+ a, a> 0$, we have
\begin{equation}\label{a1}
\begin{split}
|\varkappa_{n+1}(\xi)-S(\varkappa_{n+1}(\xi),&\varkappa_{n}(\xi))|^p+|\varkappa_{n}(\xi)-S(\varkappa_{n+1}(\xi),\varkappa_{n}(\xi))|^p\\
&= \Delta_{n+1}+a^p\\
&= \left|\varkappa_{n+1}(\xi) - \varkappa_{n}(\xi)-a\right|^p+a^p\\
&= \left|(\mathcal{G}K+I)\varkappa_{n}(\xi) - (\mathcal{G}K+I)\varkappa_{n-1}(\xi)-a\right|^p+a^p  \\
&= \left|\mathcal{G}K\varkappa_n(\xi) - \mathcal{G}K\varkappa_{n-1}(\xi) + \varkappa_n(\xi) - \varkappa_{n-1}(\xi)-a\right|^p+a^p \\
&= \left| \frac{\mathcal{G}K\varkappa_n(\xi) - \mathcal{G}K\varkappa_{n-1}(\xi)}{\varkappa_n(\xi) - \varkappa_{n-1}(\xi)-a} +1 \right|^p 
\left|\varkappa_n(\xi) - \varkappa_{n-1}(\xi)-a\right|^p+a^p. 
\end{split}
\end{equation}

Consider
\[
\frac{\eth(\tan \sigma)}{d\varkappa} = \sec^2\sigma\frac{d\sigma}{d\varkappa} = (1 + \tan^2 \sigma)\frac{d\sigma}{dt}\frac{dt}{d\varkappa}.
\]

This can be rewritten as:
\[
\left(1 + \left(\frac{dh}{d\xi} \right)^2\right)
\left(\frac{d\sigma}{dt}\right)\left(\frac{dt}{d\varkappa}\right).
\]

Hence, the quantity $\frac{dh}{d\xi}$ remains bounded due to the restrictions on the aircraft's acceleration, while $\frac{d\sigma}{dt}$ and $\frac{d\varkappa}{dt}$ are related through Eq.~\eqref{aa4}, representing the governing dynamics of the aircraft.

Now, we will show that \[
\frac{\mathcal{G}K\varkappa_n(\xi) - \mathcal{G}K\varkappa_{n-1}(\xi)}{\varkappa_n - \varkappa_{n-1}-a} > -1 - e^{\frac{1}{p}(\Delta_n-\Delta_{n+1}-\omega  )}.
\]
From condition (F$_1$), we have
\begin{equation}\label{a3}
\begin{split}
 &(\xi - \xi_0)S\delta \le 1 + e^{\frac{1}{p}(\Delta_n-\Delta_{n+1}-\omega  )} \\
&\Rightarrow  -(\xi - \xi_0)S\delta >- 1 - e^{\frac{1}{p}(\Delta_n-\Delta_{n+1}-\omega  )}. 
\end{split}
\end{equation}

Let
\begin{equation*}
\frac{\eth(\tan \sigma)}{d\varkappa} \ge -\delta , 
\end{equation*}
where $\delta$ is a positive constant, which has not been specified yet.

Now, consider the difference between $\tan \sigma_n$ and $\tan (\sigma_{n-1}+a)$:
\[
\tan \sigma_n - \tan (\sigma_{n-1}+a) = \int_{\varkappa_{n-1}+a}^{\varkappa_n} \frac{\eth(\tan \sigma)}{d\varkappa} d\varkappa \ge -\eth(\varkappa_n-\varkappa_{n-1}-a).
\]

Therefore, we have:
\begin{equation}\label{a4}
\frac{S \int_{\xi_0}^{\xi } (\tan \sigma_n - \tan (\sigma_{n-1}+a)) d\xi} {\varkappa_n - \varkappa_{n-1}-a} > - (\xi - \xi_0)S\delta.
\end{equation}

From \eqref{a3} and \eqref{a4},
\begin{equation}\label{a5}
\frac{S \int_{\xi_0}^{\xi } (\tan \sigma_n - \tan (\sigma_{n-1}+a)) d\xi} {\varkappa_n - \varkappa_{n-1}-a}\ge -S\delta(\xi-\xi_0) > - 1 - e^{\frac{1}{p}(\Delta_n-\Delta_{n+1}-\omega  )}. 
\end{equation}

However,
\[
\mathcal{G}K\varkappa_n(\xi) - \mathcal{G}K\varkappa_{n-1}(\xi)
= S \int_{\xi_0}^{\xi } (\tan \sigma_n - \tan \sigma_{n-1}) d\xi \ge S \int_{\xi_0}^{\xi } (\tan \sigma_n - \tan (\sigma_{n-1}+a)) d\xi.
\]

Therefore, \eqref{a5} becomes
\[
\frac{\mathcal{G}K\varkappa_n(\xi) - \mathcal{G}K\varkappa_{n-1}(\xi)}{\varkappa_n - \varkappa_{n-1}-a} >  - 1 - e^{\frac{1}{p}(\Delta_n-\Delta_{n+1}-\omega  )}.
\]

From physical considerations, $\varkappa(\xi)$ is inversely proportional to $\sigma(\xi)$. Hence, as per the assumption $(F_2)$,
\[
\dfrac{\mathcal{G}K\varkappa_{n}(\xi) - \mathcal{G}K\varkappa_{n-1}(\xi)}{\varkappa_n - \varkappa_{n-1}-a} < -1 + e^{\frac{1}{p}(\Delta_n-\Delta_{n+1}-\omega  )}.
\]

Thus from the above two inequalities,
\[
-1 - e^{\frac{1}{p}(\Delta_n-\Delta_{n+1}-\omega  )} < \frac{\mathcal{G}K\varkappa_n(\xi) - \mathcal{G}K\varkappa_{n-1}(\xi)}{\varkappa_n - \varkappa_{n-1}-a} < -1 + e^{\frac{1}{p}(\Delta_n-\Delta_{n+1}-\omega  )}.
\]

which yields that
\[
\left| \frac{\mathcal{G}K\varkappa_n(\xi) - \mathcal{G}K\varkappa_{n-1}(\xi)}{\varkappa_n - \varkappa_{n-1}-a} + 1 \right| < e^{\frac{1}{p}(\Delta_n-\Delta_{n+1}-\omega  )}
\quad \text{for all } \xi.
\]
Raising both sides to the power $p\ge1$, we have
\[
\left| \frac{\mathcal{G}K\varkappa_n(\xi) - \mathcal{G}K\varkappa_{n-1}(\xi)}{\varkappa_n - \varkappa_{n-1}-a} + 1 \right|^p < e^{(\Delta_n-\Delta_{n+1}-\omega  )}
\quad \text{for all } \xi.
\]
Thus, from Eq. \eqref{a1},
\[
\Delta_{n+1}+a^p < e^{(\Delta_n-\Delta_{n+1}-\omega  )} \Delta_n+a^p.
\]
\[\Rightarrow \max_{\xi }\left(\Delta_{n+1}+a^p\right)e^{\max_{\xi }(\Delta_{n+1}-\Delta_n)} < e^{-\omega  } \max_{\xi }\Delta_n+a^p.\]
\[\Rightarrow \frac{\max_{\xi }\Delta_{n+1}+a^p}{\max_{\xi }\Delta_n+a^p}\cdot\frac{e^{\max_{\xi }(\Delta_{n+1})}}{e^{\max_{\xi }\Delta_n}} < e^{-\omega  } .\]

By Lemma \ref{remark1}, this satisfies the conditions of Theorem \ref{sf theorem}, when $F(\xi) = \log \xi+\xi$.

Hence \eqref{aa2} converges if $(F_1)$ and $(F_2)$ are satisfied.
\end{proof}

\section{Conclusion}\label{sec13}

In this paper, we combine the concepts of \(F\)-contraction and \(S^{B}\)-contraction in the setting of super-metric spaces. In particular, we introduce the notions of \(S^{F}\)-contraction and Bianchini \(S^{F}\)-contraction. We show that the class of \(S^{F}\)-contractions properly contains the class of \(S^{B}\)-contractions, while the converse does not hold. Moreover, we prove that every mapping satisfying the \(S^{F}\)-contraction condition in a complete super-metric space is a \textit{Picard operator}.

Furthermore, we establish that the class of \textit{Kannan \(S^{F}\)-contractions} properly contains the class of \textit{Kannan \(F\)-contractions}, but not conversely. In addition, we prove that the class of \textit{Bianchini \(S^{F}\)-contractions} properly contains the class of \ textit{Kannan \(S^{F}\)-contractions}, whereas the reverse inclusion does not hold. Finally, we apply our results to a model describing an airplane capable of automatically following the terrain.

\backmatter

\section*{Declaration}
\begin{itemize}
    \item \textbf{Funding:} The authors did not receive any external funding.
    \item \textbf{Conflict of Interest:} The authors declare that they have no conflict of interest.
    \item\textbf{Ethical approval:} This article does not contain any studies with human participants or animals performed by any of the authors.
    \item \textbf{Data availability: } All necessary data are included into the paper.
    \item \textbf{Authors' contribution: } Irom Shashikanta Singh conducted the analysis and developed the proofs for the main results, taking a leading role in writing the manuscript. Yumnam Mahendra Singh contributed by verifying the proofs and reviewing the manuscript for typographical and formatting accuracy. Both authors have read and approved the final version of the manuscript.
\end{itemize}






\end{document}